\documentclass[12pt]{article}
\usepackage{amssymb}
\usepackage[top=50pt,bottom=60pt,left=78pt,right=76pt]{geometry}
\usepackage{graphicx, amsmath, amsfonts, amsthm, amssymb, latexsym, verbatim,
pifont, xspace, mathrsfs, enumerate}

\newtheorem{theorem}{Theorem}[section]
\newtheorem{lemma}[theorem]{Lemma}
\newtheorem{prop}[theorem]{Proposition}

\usepackage{hyperref}
\newcommand{\Tr}{Tr}
\newcommand{\Z}{\mathbb{Z}}

\begin{document}
\title{Association schemes with given stratum dimensions:
on a paper of Peter M. Neumann}
\author{Marina Anagnostopoulou-Merkouri and Peter J. Cameron\\
\small{School of Mathematics and Statistics, University of St Andrews, 
St Andrews, Fife KY16 9SS, UK}}
\date{}
\maketitle

\begin{center}
In memory of Peter Neumann: teacher, colleague, friend
\end{center}

\begin{abstract}
In January 1969, Peter M. Neumann wrote a paper entitled ``Primitive
permutation groups of degree $3p$''. The main theorem placed restrictions
on the parameters of a primitive but not $2$-transitive permutation group
of degree three times a prime. The paper was never published, and the results
have been superseded by stronger theorems depending on the classification of
the finite simple groups, for example a classification of primitive groups of
odd degree.

However, there are further reasons for being interested in this paper. First,
it was written at a time when combinatorial techniques were being introduced
into the theory of finite permutation groups, and the paper gives a very 
good summary and application of these techniques. Second, like its predecessor
by Helmut Wielandt on primitive groups of degree $2p$, it can be 
re-interpreted as a combinatorial result concerning association schemes whose
common eigenspaces have dimensions of a rather limited form. This result 
uses neither the primality of $p$ nor the existence of a permutation group
related to the combinatorial structure. We extract these results and give
details of the related combinatorics.
\end{abstract}

\section{Introduction}

In 1956, Helmut Wielandt~\cite{wielandt:2p} proved the following result:

\begin{theorem}
Let $G$ be a primitive permutation group of degree $2p$, where $p$ is prime.
If $G$ is not $2$-transitive, then $n=2a^2+2a+1$ for some positive integer
$a$, and $G$ has rank~$3$ and subdegrees $a(2a+1)$ and $(a+1)(2a+1)$.
\end{theorem}

The proof of this theorem is also given in Chapter~$5$ of his
book~\cite{wielandt:book}. It illustrates an extension of the methods of
Schur rings using representation theory. He mentioned that, for $a=1$, we
have two examples: the groups $S_5$ and $A_5$, acting on the set of 
$2$-element subsets of $\{1,\ldots,5\}$.

Now it is possible to show that there are no others. For example, using the
Classification of Finite Simple Groups, all the finite primitive rank~$3$
permutation groups have been determined~\cite{kl,l:rk3,ls:rk3}, and the 
observation can be verified by checking the list.

However, there is more to be said. Wielandt's proof falls into two parts.
The first involves showing that the permutation character of $G$ decomposes
as $1_G+\chi_1+\chi_2$, where $1_G$ is the principal character of $G$ and
$\chi_1,\chi_2$ are irreducibles with degrees $p-1$ and $p$. It follows from
this that $G$ has rank~$3$ and is contained in the automorphism group of a
strongly regular graph, having the property that the eigenvalues of its
adjacency matrix have multiplicities $1$, $p-1$, and $p$. Now the argument
shows something much more general. Neither the existence of a rank~$3$ group
of automorpisms nor the primality of $p$ are needed.

First, a definition: a graph $\Gamma$ is \emph{strongly regular} with
parameters $(n,k,\lambda,\mu)$ if it has $n$ vertices, every vertex has $k$
neighbours, and two vertices have $\lambda$ or $\mu$ common neighbours
according as they are joined by an edge or not. Every rank~$3$ group of even
order is the automorphism group of a strongly regular graph, but not
conversely; many strongly regular graphs have no non-trivial automorphisms.
Any regular graph has the all-$1$ vector as an eigenvector; a regular graph is
strongly regular if and only if its adjacency matrix, acting on the space
orthogonal to the all-$1$ vector, has just two eigenvalues.

\begin{theorem}
Let $\Gamma$ be a strongly regular graph on $2n$ vertices, with the property
that the eigenvalues of the adjacency matrix, on the space of vectors 
orthogonal to the all-$1$ vector, have dimensions $n-1$ and $n$. Then either
\begin{enumerate}
\item $\Gamma$ is a disjoint union of $n$ complete graphs of size $2$, or
the complement of this; or
\item for some positive integer $a$, we have $n=2a^2+2a+1$, and up to
complementation the parameters of the graph $\Gamma$ are given by
\[n=(2a+1)^2+1,\quad k=a(2a+1),\quad \lambda=a(a+2),\quad \mu=(a+1)^2.\]
\end{enumerate}
\end{theorem}

We are not aware of who first pointed this out. The result is given, for
example, as Theorem~2.20 in~\cite{cvl}.

In the case $a=1$, the complementary strongly regular graphs are the line
graph of the complete graph $K_5$ and the Petersen graph. But, unlike in
Wielandt's case, there are many others. For example, suppose that there exists
a Steiner system $S(2,a+1, 2a^2+2a+1)$. Then the strongly regular graph
whose vertices are the blocks, two vertices adjacent if the corresponding
blocks intersect, has the parameters given in the theorem. For example, when
$a=2$, the two Steiner triple systems on $13$ points give non-isomorphic
strongly regular graphs on $26$ vertices. (We discuss examples further in the
last section.)

\medskip

Now to the subject of this paper. In 1969, Peter Neumann wrote a long paper~\cite{pmn} extending Wielandt's result from $2p$ to $3p$, where $p$ is prime.
His conclusion is that, if such a group is not $2$-transitive, then $p$ is given
by one of three quadratic expressions in a positive integer $a$, or one of
three sporadic values; the rank is at most~$4$, and the subdegrees are given
in each case.

Like Wielandt's, Neumann's proof falls into two parts: first find the
decomposition of the permutation character, and then in each case find
the combinatorial implications for the structure acted on by the group. 
In contrast to Wielandt, the first part is much easier, since in the intervening
time, Feit~\cite{feit} had given a characterisation of groups with order
divisible by $p$ having a faithful irreducible representation of degree less
than $p-1$. On the other hand, the second part is much harder; rather than 
just one possible decomposition of the permutation character, he finds eight
potential decompositions, some of which require many pages of argument.

Again like Wielandt's, Neumann's conclusions have been superseded by results
obtained using the classification of finite simple groups. For example, all
the primitive permutation groups of odd degree have been
classified~\cite{k:odd,ls:odd}.

The paper was never published. It happened that both Leonard Scott and Olaf
Tamaschke had produced similiar results. There was a plan for Neumann and
Scott to collaborate on a joint paper, but for unknown reasons this never
happened. The authors are grateful to Leonard Scott~\cite{scott:pc} for
providing a scan of Peter Neumann's original typescript together with some
historical material about the proposed collaboration. The second author has
re-typed the paper and posted it on the arXiv~\cite{pmn2}.

Our task is to produce a combinatorial version of this, as we have seen for
Wielandt's theorem. We give some historical background to the theorem with
some comments on the place of Neumann's paper in the introduction of 
combinatorial methods into the study of permutation groups, and to check in
detail that his arguments give combinatorial results which do not depend
on either the existence of a primitive group or the primality of $p$. Indeed
we find some families of parameters which do not occur in Neumann's case
since the number of vertices is even.

\section{History}

The 1960s saw a unification of combinatorial ideas which had been developed
independently in three different areas of mathematics. In statistics,
R~C.~Bose and his colleagues and students developed the concept of an
\emph{association scheme}. Extracting information from experimental results
requires inversion of a large matrix, and Bose realised that the task would be
much simpler if the matrix belonged to a low-dimensional subalgebra of the
matrix algebra; requiring entries to be constant on the classes of an
association scheme achieves this. In the former Soviet Union, Boris Weisfeiler
and his colleagues were studying the graph isomorphism problem, and developed
the concept of a \emph{cellular algebra}, an isomorphism invariant of graphs,
to simplify the problem, and an algorithm, the \emph{Weisfeiler--Leman
algorithm}, to construct it. In Germany, Helmut Wielandt was extending the
method of \emph{Schur rings} to study permutation groups with a regular
subgroup; by using methods from representation theory he was able to dispense
with the need for the regular subgroup. These techniques were further developed
by Donald Higman in the USA, under the name \emph{coherent configuration}.

The three concepts are very closely related. We begin with Higman's definition.
A \emph{coherent configuration} consists of a set $\Omega$ together with a
set $\{R_1,R_2,\ldots,R_r\}$ of binary relations on $\Omega$ with the properties
\begin{enumerate}
\item $\{R_1,\ldots,R_r\}$ form a partition of $\Omega\times\Omega$;
\item there is a subset of $R_1,\ldots,R_r$ which is a partition of the
\emph{diagonal} $\{(\omega,\omega):\omega\in\Omega\}$ of $\Omega^2$;
\item the converse of each relation $R_i$ is another relation in the set;
\item for any triple $(i,j,k)$ of indices, and any $(\alpha,\beta)\in R_k$,
the number $p_{ij}^k$ of $\gamma\in\Omega$ such that $(\alpha,\gamma)\in R_i$
and $(\gamma,\beta)\in R_j$ depends only on $(i,j,k)$ and not on the choice
of $(\alpha,\beta)\in R_k$.
\end{enumerate}
The number $r$ is the \emph{rank} of the configuration.  
Combinatorially, a coherent configuration is a partition of the edge set of
the complete directed graph with loops.

A coherent configuration is \emph{homogeneous} if the diagonal is a single
relation. In the group case, this means that the group is transitive. All
the configurations in this paper will be homogeneous.

If $G$ is a permutation group on $\Omega$, and we take the relations $R_i$
to be the orbits of $G$ on $\Omega^2$, we obtain a coherent configuration.
This was Higman's motivating example, which he called the \emph{group case}.
Not every coherent configuration falls into the group case; indeed, our
task is to extend Neumann's results from the group case to the general case.

The notion of a cellular algebra is the same apart from an inessential small
difference (the diagonal is replaced by some equivalence relation). 
Association schemes form a special case, where all the relations $R_i$ are
symmetric. It follows that, in an association scheme, the diagonal is a single
relation. (Statisticians deal with symmetric matrices, for example
covariance matrices.)

A coherent configuration with rank $2$ is \emph{trivial}: one relation is
the diagonal, the other is everything else. For rank~$3$, we can suppose
without loss that $R_1$ is the diagonal. There are then two possibilities:
\begin{itemize}
\item $R_3$ is the converse of $R_2$. Then $R_2$ is a \emph{tournament}
(an orientation of the edges of the complete graph on $\Omega$); condition
(d) shows that it is a \emph{doubly regular} tournament~\cite{rb}.
\item $R_2$ and $R_3$ are symmetric. Then each is the edge set of a graph,
and these graphs are \emph{strongly regular}~\cite[Chapter~2]{cvl}.
\end{itemize}

The definition of coherent configuration has an algebraic interpretation.
Let $A_i$ be the \emph{adjacency matrix} of the relation $R_i$, the
$\Omega\times\Omega$ matrix with $(\alpha,\beta)$ entry $1$ if
$(\alpha,\beta)\in R_i$. Then $A_1,\ldots,A_r$ are zero-one matrices
satisfying the following conditions:
\begin{enumerate}
\item $A_1+\cdots+A_r=J$, the all-$1$ matrix;
\item there is a subset of these matrices whose sum is the identity $I$;
\item for any $i$ there is a $j$ such that $A_i^\top=A_j$;
\item $\displaystyle{A_iA_j=\sum_{k=1}^rp_{ij}^kA_k}$.
\end{enumerate}
Condition (d) says that the linear span over $\mathbb{C}$ of $A_1,\ldots,A_r$ 
is an algebra (closed under multiplication), and condition (c) implies that
this algebra is semi-simple. In the group case, it is the \emph{centraliser
algebra} of the permutation group, consisting of matrices which commute with
every permutation matrix in the group. In the case of association schemes, it
is known as the \emph{Bose--Mesner algebra} of the scheme. In this case, all
the matrices are symmetric, the algebra is commutative, and we can work over
$\mathbb{R}$. In the group case, the centraliser algebra is commutative if
and only if the permutation character is multiplicity-free.

If the algebra is commutative, then the matrices are simultaneously 
diagonalisable; the common eigenspaces are called the \emph{strata} of
the configuration. In the rank~$3$ case where we have a strongly regular
graph and its complement, the stratum dimensions are simply the multiplicities
of the eigenvalues. We occasionally extend the use of the word ``stratum''
to the non-commutative case, where it means a submodule for the algebra
spanned by the matrices which is maximal with respect to being a sum of
isomorphic submodules.

In all cases which arise in Peter Neumann's paper, the algebra turns out to
be commutative, although there are two potential cases where the permutation
character is not multiplicity-free; both of these are eliminated.

\medskip

It seems clear to the authors that, had the paper been published in 1969, it
would have been very influential: it provides both a clear account of the
theory and how it can be used to study permutation groups, and also a
non-trivial example of such an application. The second author of the present
paper read it at the start of his DPhil studies in Oxford under Peter Neumann's
supervision, and considers himself fortunate to have been given such a good
grounding in this area; he has worked on the interface of group theory and
combinatorics ever since.

\section{The results}

The main theorems in this paper are the following. They are numbered to
correspond to the eight cases in Neumann's paper.

\begin{theorem}\label{thm-typeI}
Let $\mathcal{A} = \{I_n, A_1, A_2\}$ be a coherent configuration of $n \times n$ matrices. If the eigenvalues of $A_1$ have multiplicities $1, \frac{n - 1}{2}, \frac{n - 1}{2}$ then one of the two following cases must hold:
\begin{itemize}
    \item $n \equiv 1 \pmod{4}$ and $A_1$ and $A_2$ are the adjacency matrices of conference graphs;
    \item $n \equiv 3 \pmod{4}$ and $A_1$ and $A_2$ are the adjacency matrices of doubly regular tournaments.
\end{itemize}
\end{theorem}

\begin{theorem}\label{thm-typeII}
Let $G$ be a strongly regular graph on $3n$ vertices. If the multiplicities of the eigenvalues of $G$ are $1, n, 2n - 1$ then $G$ or its complement have the following parameters in terms of a non-negative integer $a$:
\begin{itemize}
    \item $3n = 144a^2 + 54a + 6$, $k_1 = 48a^2 + 14a + 1$, $\lambda = 16a^2 + 6a, \mu = 16a^2 + 2a$;
    \item $3n = 144a^2 + 90a + 15$, $k_1 = 48a^2 + 34a + 6, \lambda = 16a^2 + 10a + 1, \mu = 16a^2 + 14a + 3$;
    \item $3n = 144a^2 + 198a + 69$, $k_1 = 48a^2 + 62a + 20, \lambda = 16a^2 + 22a + 7, \mu = 16a^2 + 18a + 5$;
    \item $3n = 144a^2 + 234a + 96$, $k_1 = 48a^2 + 82a + 35, \lambda = 16a^2 + 26a + 10, \mu = 16a^2 + 30a + 14$.
\end{itemize}
\end{theorem}

\begin{theorem}\label{thm-typeIII}
Let $G$ be a strongly regular graph on $3n$ vertices. If the multiplicities of the eigenvalues of $G$ are $1, 2n, n - 1$ then either $G$ or its complement is a disjoint union of $n$ copies of $K_3$ or $G$ or its complement have the following parameters for some non-negative integer $a$:
\begin{itemize}
    \item $3n = 9a^2 + 9a + 3$, $k_1 = 3a^2 + 5a + 2, \lambda = a^2 + 3a + 1, \mu = (a + 1)^2$;
    \item $3n = 9a^2 + 9a + 3$, $k_1 = 3a^2 + a, \lambda = a^2 - a - 1, \mu = a^2$.
\end{itemize}
\end{theorem}

\begin{theorem}\label{thm-typeIV}
Let $\mathcal{A} = \{I_{3n}, A_1, A_2, A_3\}$ be a coherent configuration of $3n \times 3n$ matrices. If the multiplicities of the eigenvalues of $A_1, A_2, A_3$ are $1, n, n, n-1$ then one of the following hold:
\begin{itemize}
\item $A_2 = A_3^T$ and the row sums of $A_1, A_2,$, and $A_3$ are $n - 2a - 1, n + a$, and $n + a$ respectively for some even integer $a$;
\item $A_2 = A_3^T$ and the row sums of $A_1, A_2$, and $A_3$ are $n + 2a + 1, n - a - 1$, and $n - a - 1$ respectively for some odd integer $a$;
 \item All matrices are symmetric and the row sums of $A_1, A_2, A_3$ are $n + 2a + 1, n - a - 1$, and $n - a - 1$  respectively for some non-negative integer $a$.
 \end{itemize}
\end{theorem}

\begin{theorem}\label{thm-typeV}
There exists no coherent configuration $\mathcal{A} = \{I_{3n}, A_1, A_2, A_3, A_4, A_5\}$ of $3n\times 3n$ matrices such that the multiplicities of the eigenvalues of $A_1, \ldots, A_5$ are $1, n, n, n-1$.
\end{theorem}

\begin{theorem}\label{thm-typeVI}
There is no strongly regular graph on $3n$ vertices with eigenvalue multiplicities $1, n + 1, 2(n - 1)$. 
\end{theorem}

\begin{theorem}\label{thm-typeVII}
Let $\mathcal{A} = \{I_{3n}, A_1, A_2, A_3\}$ be a coherent configuration of $3n\times 3n$ matrices. If the eigenvalues of $A_1, \ldots, A_3$ have multiplicities $1, n + 1, n - 1, n - 1$, then $\mathcal{A}$ is an association scheme and one of the following hold:
\begin{itemize}
    \item $n = 7$ and the row sums of $A_1, A_2, A_3$ are $4, 8$, and $8$;
    \item $n = 19$ and the row sums of $A_1, A_2, A_3$ are $6, 20$, and $30$;
    \item $n = 31$ and the row sums of $A_1, A_2, A_3$ are $32, 40$, and $20$.
\end{itemize}
\end{theorem}

\begin{theorem}\label{thm-typeVIII}
There exists no coherent configuration $\mathcal{A} = \{I_{3n}, A_1, A_2, A_3, A_4, A_5\}$ of $3n\times 3n$ matrices, where $A_1, \ldots, A_5$ have eigenvalues with multiplicities $1, n + 1, n - 1, n- 1$.
\end{theorem}

\section{The proofs}

\subsection{A lemma}

We start with a lemma that will be used throughout the paper.

%

\begin{lemma}
Let $\mathcal{A}$ be a homogeneous coherent configuration on $n$ points.
Suppose that the dimension of a non-trivial stratum for $\mathcal{A}$ is
at least $n/3-1$. Then one of the following happens:
\begin{enumerate}
\item One of the relations in $\mathcal{A}$ has at least $n/3$ connected
components.
\item Any matrix in $\mathcal{A}$ has the property that any eigenvalue
$\lambda$ apart from the row sum $r$ satisfies $|\lambda|<r$.
\end{enumerate}
\label{lemma-perron-frob}
\end{lemma}

\begin{proof}
We use the \emph{Perron--Frobenius Theorem}, see~\cite{gantmacher}. For
any non-negative matrix $A$, one of the following holds:
\begin{itemize}
\item Under simultaneous row and column permutations, $A$ is equivalent to
a matrix of the form $\begin{pmatrix}B&O\\O&C\\\end{pmatrix}$. In our case
the constancy of the row sum $r$ means that $r$ has multiplicity equal to
the number of connected components; so there are at least $n/3$ connected
components, and (a) holds.
\item $A$ is decomposable, that is, under simultaneous row and column
permutations it is equivalent to a matrix of the form
$\begin{pmatrix}B&X\\O&C\\\end{pmatrix}$, where $X\ne O$. But this
contradicts the fact that the row sum is constant.
\item $A$ is imprimitive, that is, equivalent under simultaneous row and
column permutations to a matrix of the form
\[\begin{pmatrix}O&B_1&\ldots&\ldots&0\\ O&O&B_2&\ldots&O\\
\ldots&\ldots&\ldots&\ldots&\ldots\\ B_t&O&O&\ldots&O\end{pmatrix}.\]
But then $r\mathrm{e}^{2\pi\mathrm{i}k/t}$ is a simple eigenvalue for
$k=0,1,\ldots,t-1$, contrary to assumption.
\item $A$ is primitive. Then the Perron--Frobenius Theorem asserts that
there is a single eigenvalue with largest absolute value, as required.
\end{itemize}
\end{proof}

\subsection{Proof of Theorem~\ref{thm-typeI}}

We first prove a lemma about strongly regular graphs that will be used in the proof of Theorem~\ref{thm-typeI}.

\begin{lemma}\label{lemma-conference}
Let $G$ be a strongly regular graph with parameters $(n, k, \lambda, \mu)$ and let $k, r, s$ be the eigenvalues of the adjacency matrix of $G$. If $r$ and $s$ have equal multiplicities then $G$ is a conference graph.
\end{lemma}

\begin{proof}
It is known for a strongly regular graphs that the multiplicities of $r$ and $s$ are 
\[f, g = \frac{1}{2}( n - 1 \pm \frac{(n - 1)(\mu - \lambda) - 2k}{\sqrt{(\mu - \lambda)^2 - 4(k - \mu)}})
\]
respectively. Hence, if $f = g$ then it follows that 

\[
(n - 1)(\mu - \lambda) - 2k = -(n-1)(\mu - \lambda) +2k \Rightarrow 2k = (n - 1)(\mu - \lambda)
\]
and thus $G$ is a conference graph, as required. Moreover, $f = g = \frac{n - 1}{2}$.
\end{proof}

\begin{proof}[Proof of Theorem~\ref{thm-typeI}]
Since $\mathcal{A}$ is a coherent configuration, $A_0 + A_1 + A_2 = J_{n}$ and moreover $A_{i}^{T} = A_{j}$ for $i, j\in \{1, 2\}$. Hence, there are two possibilities. Either $A_i = A_i^{T}$ for $i \in \{1, 2\}$ or $A_i = A_j^T$ for $i, j\in \{1, 2\}$ and $i\neq j$. 

In the first case, the graphs with adjacency matrices $A_1$ and $A_2$ are undirected. Moreover, since $A_1$ and $A_2$ are symmetric, $\mathcal{A}$ is an association scheme and hence those graphs are strongly regular and one is the complement of the other. It follows by Lemma~\ref{lemma-conference} that $A_1$ and $A_2$ are the adjacency matrices of conference graphs and in fact two copies of the same conference graph. Moreover, for a conference graph to exist, it is known that $n\equiv 1\pmod{4}$.

In the second case, since $\mathcal{A}$ is a coherent configuration, it follows that $A_1$ and $A_2$ must have constant row and column sums and hence their digraphs are regular. Let $G_1, G_2$ be the digraphs with adjacency matrices $A_1$ and $A_2$ respectively and $V$ be the vertex set of those digraphs. For $u, v\in V$, we write $u \rightarrow_{G_1} v$ if $v$ is an out-neighbour of $u$ in $G_1$ and similarly $u \rightarrow_{G_2} v$ if $v$ is an out-neighbour of $u$ in $G_2$. Since $A_1 + A_2 = J - I$, it follows that $u \rightarrow_{G_1} v \iff u\not\rightarrow_{G_2} v$ and vice versa and also that either $(A_k)_{ij} = 1$ or $(A_k)_{ji} = 1$ for $k\in \{1, 2\}$. Hence, $G_1$ and $G_2$ are regular tournaments. Also, notice that since $\mathcal{A}$ is a coherent configuration, it follows that for $m, n\in \{1, 2\}, m\neq n$, there exists a constant $p_{mn}^m$ such that for any $i, j\in V$, such that $(A_m)_{ij} = 1$, $|\{k \mid (A_m)_{ik} = 1, (A_n)_{kj} = 1\}| = |\{k \mid (A_m)_{ij} = 1, (A_m)_{jk} = 1\}| = p_{mn}^m$. Hence, both $G_1$ and $G_2$ are doubly regular, and it is known that $n\equiv 3 \pmod{4}$ for doubly regular tournaments.
\end{proof}

\subsection{Proof of Theorem~\ref{thm-typeII}}

\begin{proof}
Let $A_1$ be the adjacency matrix of $G$ and $A_2$ be the adjacency matrix of its complement. Since $G$ is strongly regular, the eigenvalues of $A_1$ and $A_2$ have the same multiplicities. Moreover, if $A_1$ has eigenvalues $k_1, r_1, s_1$ then $A_2$ has eigenvalues $k_2 = 3n - k_1 - 1, r_2 = -1 -r_1, s_2 = -1 - s_1$. We know that for $i\in \{1, 2\}$
\[
\Tr (A_i) = k_i + nr_i + (2n - 1)s_i = 0
\]
Reducing modulo $n$ gives that $k_i \equiv s_i \pmod{n}$. Therefore, since by Lemma~\ref{lemma-perron-frob} $k_i > s_i$, it follows that $k_i - s_i = \epsilon_in$ for $\epsilon_i \in \{1, 2\}$. Therefore, 
\[n_1 + n_2 - s_1 - s_2 = (\epsilon_1 + \epsilon_2)n \Rightarrow 3n - 1 - s_1 +1 + s_1 = (\epsilon_1 + \epsilon_2)n \Rightarrow \epsilon_1 + \epsilon_2 = 3.\]
Assume without loss of generality that $\epsilon_1 = 1$ and $\epsilon_2 = 2$. Then, $k_1 = n + s_1$ and also
\[
n + s_1 + nr_1 + (2r - 1)s_1 = 0
\Rightarrow r_1 = -1 -2s_1.\]
Also, we have that 
\[
\Tr (A_1^2) = k_1^2 + nr_1^2 + (2n - 1)s_1^2 = 3nk_1.
\]
Appropriate substitution gives
\[
(n + s_1)^2 + n(1 + 2s_1)^2 + (2n-1)s_1^2 = 3n(n + s_1)
\]
which simplifies to
\[
6s_1^2 + 3s_1 + 1 - 2n = 0.
\]
Therefore, 
\[
s_1 = \frac{1}{4}\left(  -1 \pm \sqrt{\frac{16n - 5}{3}}\right)
\]

Since $G$ is strongly regular and its eigenvalues have different multiplicities, it is not a conference graph, and hence its eigenvalues are integer. Hence, $16n - 5 = 3b^2$ for some non-negative integer $b$. This gives us that $3b^2 + 5 \equiv 0 \pmod{16}$. It follows that $b = 3, 5, 11$ or $13 \pmod{16}$. We therefore need to examine the following four cases:\\

\medskip

\noindent\textbf{Case 1:} $b = 16a + 3$.

\medskip In this case we get:

\[
16n = 3(16a + 3)^2 + 5 \Rightarrow n = 48a^2 + 18a + 2.
\]
and $s_1 = -4a - 1$. Notice that only the negative solution works, since $16a + 2$ is not divisible by 4.
Consequently $k = 48a^2 + 14a + 1$.
We also get $r_1 = 8a + 1$

Now, using the formulae for the eigenvalues of strongly regular graphs, namely
\[
r_1, s_1 = \frac{1}{2}\left( (\lambda - \mu) \pm \sqrt{(\lambda - \mu)^2 + 4(k_1 - \mu)}\right)
\]
we get 
\begin{align*}
    \lambda - \mu = r_1 + s_1\\
    4\mu = (\lambda - \mu)^2 - (r - s)^2 + 4k.
\end{align*}
Solving this system we obtain $\lambda = 16a^2 + 6a$ and $\mu = 16a^2 + 2a$.

\medskip

\noindent\textbf{Case 2:} $b = 16a + 5$.

\medskip 

In this case we get:
\[
16n = 3(16a + 5)^2 + 5 \Rightarrow n = 48a^2 + 3a + 5.
\]
and $s_1 = 4a + 1$.
Hence, $k = 48a^2 + 7a + 6$.
We also get $r_1 = -8a - 3$\\
As above, knowing $r_1, s_1$ we can obtain $\lambda$ and $\mu$ which in this case are equal to $16a^2 + 10a + 1$ and $16a^2 + 14a + 3$ respectively.

\medskip

\noindent\textbf{Case 3:} $b = 16a + 11$.

\medskip In this case we get:
\[
16n = 3(16a + 11)^2 + 5 \Rightarrow n = 48a^2 + 66a + 23.
\]
and $s_1 = -4a -3$.
Hence, $k = 48a^2 + 62a + 20$.
Also, $r_1 = 8a + 5$ and routine calculation as above gives $\lambda = 16a^2 + 22a + 7, \mu = 16a^2 + 18a + 5$.

\medskip

\noindent\textbf{Case 4:} $b = 16a + 13$.

\medskip In this case we get:
\[
16n = 3(16a + 13)^2 + 5 \Rightarrow n = 48a^2 + 78a + 32.
\]
and $s_1 = 4a + 3$.
Hence, $k = 48a^2 + 82a + 35$, $r_1 = -8a - 7$, $\lambda = 16a^2 + 26a + 10, \mu = 16a^2 + 20a + 14$.
\end{proof}

\subsection{Proof of Theorem~\ref{thm-typeIII}}

\begin{proof}
Let $A_1$ be the adjacency matrix of $G$ and $A_2$ be the adjacency matrix of its complement. Since $G$ is strongly regular we know that the eigenvalues of $A_1$ and $A_2$ have the same multiplicities. Also, if $A_1$ has eigenvalues $k_1, r_1, s_1$, then $A_2$ has eigenvalues $k_2 = 3n - k_1 - 1, r_2 = -1 - r_1, s_2 = -1 -s_1$. We know that for $i\in \{1, 2\}$
\[
\Tr(A_i) = k_i + 2nr_i + (n - 1)s_i = 0.
\]
Reducing modulo $n$ gives that $k_i \equiv s_i \pmod{n}$, and since by Lemma~\ref{lemma-perron-frob} either one of $A_1, A_2$ is the disjoint union of $n$ copies of $K_3$ or $k_i > |s_i|$. In the second case, it follows that $k_i-s-i = \epsilon_in$ for $\epsilon_i \in \{1, 2\}$. Also, as before, $\epsilon_1 + \epsilon_2 = 3$ and hence we may suppose without loss of generality that $\epsilon_1 = 1$ and $\epsilon_2 = 2$. Then, $k_1 = n + s_1$ and $r_1 = \frac{-s_1 - 1}{2}$. We therefore get
\[
\Tr(A_1^2) = (n + s_1)^2 + 2n\left(\frac{s_1 + 1}{2}\right)^2 + (n - 1)s_1^2 = 3n(n + s_1).
\]
and simplifying gives $3s_1^2 = 4n - 1$. Therefore,
\[
s_1^2 = \frac{4n - 1}{3}.
\]
We can thus write $s_1^2$ as $(2a + 1)^2$ for some $a\geq 0$ and we get
\[
(2a + 1)^2 = \frac{4n - 1}{3} \Rightarrow n = 3a^2 + 3a + 1
\]
and $s_1 = \pm 2a + 1$. We therefore get the following cases:

\medskip

\noindent \textbf{Case 1:} $s_1 = 2a + 1$.

\medskip

In this case we get $k_1 = 3a^2 + 5a + 2$ and $r_1 = -a - 1$, and computing $\lambda$ and $\mu$ as in the proof of Theorem~\ref{thm-typeII} we obtain $\lambda = a^2 + 3a + 1$ and $\mu = (a + 1)^2$.

\medskip

\noindent \textbf{Case 2:} $s_1 = -2a - 1$.

\medskip

Here, routine calculation gives $k_1 = 3a^2 + a$, $r_1 = a$, $\lambda = a^2 - a - 1, \mu = a^2$.
\end{proof}

\subsection{Proof of Theorem~\ref{thm-typeVI}}

\begin{proof}
Suppose for a contradiction that there exists such a strongly regular graph, and let $A_1$ be its adjacency matrix and $A_2$ be the adjacency matrix of its complement and suppose that $k_1, r_1, s_1$ and $k_2, r_2, s_2$ are the eigenvalues of $A_1$ and $A_2$ respectively. Then, for $i\in \{1, 2\}$ we get

\[
\Tr(A_i) = k_i + (n + 1)r_i + 2(n - 1)s_i = 0
\]
and
\[
\Tr(A_i^2) = k_i^2 + (n + 1)r_i^2 + (2n - 1)s_i^2 = 3nk_i.
\]
Reducing modulo $n$ gives

\begin{align*}
    k_i \equiv 2s_i - r_i \pmod{n}\\
    k_i^2 \equiv 2s_i^2 - r_i^2 \pmod{n}
\end{align*}
Hence, $(2s_i - r_i)^2 \equiv 2s_i^2 - r_i^2$. By routine calculation, it follows that $s_i \equiv r_i \pmod{n}$ and consequently $k_i \equiv r_i \pmod{n}$. Therefore, $k_i = \epsilon_in + r_i$ and $s_i = \eta_in + r_i$ for some $\epsilon_i, \eta_i \in \{1, 2\}$.

Substituting into the trace equations and reducing modulo $n^2$ gives
\begin{align*}
\epsilon_in + r_i + (n + 1)r_i + 2(n - 1)r_i - 2\eta_in \equiv 0 \pmod{p^2}\\
2\epsilon_inr_i + r_i^2 + (n + 1)r_i^2 + 2(n - 1)r_i^2 - 4r_i\eta_in \equiv 3nr_i \pmod{p^2}.
\end{align*}
We now collect terms and divide by $n$ and we get
\begin{align*}
\epsilon_i + 3r_i - 2\eta_i \equiv 0 \pmod{n}\\
3r_i^2 + r_i(2\epsilon_i - 4\eta_i - 3)\equiv 0 \pmod{n}.
\end{align*}
Since $1 + r_1 + r_2 = 0$ it cannot be the case that both $r_1$ and $r_2$ are divisible by $n$. Hence, interchanging $A_1$ and $A_2$ if necessary we may assume that $r_1 \not\equiv 0 \pmod{n}$. Then,
\begin{align*}
    3r_1\equiv 2\eta_1 - \epsilon_1 \pmod{n}\\
    3r_1 \equiv 4\eta_1 - 2\epsilon_1 + 3 \pmod{n}.
\end{align*}
Eliminating $2\eta_1 - \epsilon_1$ gives $r_1 \equiv -1 \pmod{n}$. Therefore, since $k_1 \equiv r_1 \pmod{n}$, either $k_1 = n - 1$ or $k_1 = 2n - 1$. If $k_1 = n - 1$, then since $r_1 \equiv s_1 \equiv -1 \pmod{n}$ and by Lemma~\ref{lemma-perron-frob} $|r_1| < k_1$ and $|s_1| < k_1$, it follows that $r_1 = s_1 = -1$. However, by looking at the formulae for $r_1$ and $s_1$ for a strongly regular graph, we deduce that $r_1 \neq s_1$, a contradiction. Similarly, if $k_1 = 2n - 1$, then $k_2 = n$ which forces $r_2 = s_2 = 0$, again a contradiction. Hence, there is no strongly regular graph with those eigenvalue multiplicities.
\end{proof}

\subsection{Proof of Theorem~\ref{thm-typeIV}}

\begin{proof}
Let $k_i, r_i, s_i, t_i$ be the eigenvalues of $A_i$ for $i\in \{1, 2, 3\}$ with multiplicities $1, n, n, n - 1$ respectively. Firstly notice that  $t_i$ must be a rational integer and $r_i$ and $s_i$ must either both be rational integers or algebraically conjugate algebraic integers. Then, we get
\begin{align*}
\Tr(A_i) = k_i + nr_i + ns_i + (n - 1)t_i = 0\\
\end{align*}
Hence, $n$ must divide $k_i - t_i$, and since by Lemma~\ref{lemma-perron-frob} $n_i > t_i$, it follows that $k_i = \epsilon_in + t_i$ for some $\epsilon_i > 0$. Moreover, by Equation (6.9) in~\cite{pmn}, $\epsilon_1 + \epsilon_1 + \epsilon_3 = 3$ and hence $\epsilon_i = 1$ for all $i\in \{1, 2, 3\}$. Thus, $k_i = n + t_i$.

There are now two cases to consider. Either all matrices are symmetric or two of them, say $A_2$ and $A_3$ without loss of generality are such that $A_2^T = A_3$. We first consider the second case. In this case the eigenvalues of $A_2$ and $A_3$ are the same. Hence, $t_2 = t_3$ and either $r_2 = r_3$ and $s_2 = s_3$ or $r_2 = s_3$ and $r_3 = s_2$. Notice that the algebra spanned by the matrices of this coherent configuration is commutative and therefore $A_2$ and $A_3$ can be simultaneously diagonalised. Let $U$ be the matrix that simultaneously reduces $A_2$ and $A_3$. If $r_2 = r_3$ and $s_2 = s_3$ then $U^{-1}A_2U = U^{-1}A_3U$, which implies that $A_2 = A_3$, a contradiction. Hence, $r_2 = s_3$ and $r_3 = s_2$.

Now adding $A_2$ and $A_3$ together produces an association scheme of the type arising in Theorem~\ref{thm-typeIII}. Hence, $n = 3a^2 + 3a + 1$ and either $k_1 = n - 2a - 1$ and $k_2 = k_3 = n + a$ or $k_1 = n + 2a + 1$ and $k_2 = k_3 = n - a - 1$.

We now show that if $k_1 = n - 2a - 1$ then $a$ is even and if $k_1 = n + 2a + 1$ then $a$ is odd. In the first case, the remaining eigenvalues of $A_1, A_2$, and $A_3$ are as shown below:
\begin{align*}
r_1 = a, s_1 = a, t_1 = -2a - 1\\
r_2 = r, s_2 = s, t_2 = a\\
r_3 s, s_3 = r, t_3 = a.
\end{align*}
where $r + s = -a - 1$. Now Equation (6.7) in~\cite{pmn} gives
\[
rs = \frac{1}{2}(2n - a - a^2) = \frac{1}{2}(5a + 2)(a + 1)
\]
and Equation (6.8) in~\cite{pmn} gives 
\[
3n(n + a)a_{22}^3 = (n + a)^3 + nrs(r + s) + (n - 1)a^3.
\]
Eliminating $rs$ and simplifying gives $a_{22}^3 = a^3 + \frac{3a}{2}$ and since $a_{22}^3 \in \Z$, $a$ must be even.

In the second case, the eigenvalues of $A_1, A_2$, and $A_3$ are the ones given below:
\begin{align*}
r_1 = -a - 1, s_1 = -a - 1, t_1 = 2a + 1\\
r_2 = r, s_2 = s, t_2 = -a - 1\\
r_3 = s, s_3 = r, t_3 = -a -1
\end{align*}
where $r + s = 1$ by Equation (6.6) in~\cite{pmn}. Equation (6.7) in~\cite{pmn} gives $rs = \frac{1}{2}a(5a + 3)$ and from Equation (6.8) in~\cite{pmn} we get
\[
3n(n - a - 1)a_{22}^3 = (n - a - 1)^3 + nrs(r + s) - (n - 1)(a + 1)^3.
\]
Simplifying gives $a_{22}^3 = a^2 + \frac{a - 1}{2}$, and since $a_{22}^3\in \Z$, it follows that $a$ is odd, as claimed.

We  now consider the symmetric case. We get the following equations
\begin{equation}
s_i + r_i = -1 - t_i
\end{equation}
\begin{align*}
\Tr(A_i^2) = k_i^2 + nr_i + ns_i + (n - 1)t_i = 3nk_i \Rightarrow (t_i + n)^2 + nr_i^2 + ns_i^2 + nt_i^2 - t_i^2 = 3n(n + t_i) \Rightarrow
\end{align*}
\begin{equation}
r_i ^ 2 + s_i ^2 = -t_i^2 + t_i + 2n
\end{equation}
From this we get $2r_is_i = 1 + t_i + 2t_i^2 + 2n$ and hence we deduce that $s_i$ is odd. Also, we can calculate $r_i$ and $s_i$ and we find that $r_i, s_i = \frac{1}{2}(-1 - t_i \pm \sqrt{4n - 1 - 3t_i^2})$. Without loss of generality we set
\begin{align*}
    r_i = \frac{1}{2}(-1 - t_i + \sqrt{4n - 1 - 3t_i^2})\\
    s_i = \frac{1}{2}(-1 - t_i - \sqrt{4n - 1 - 3t_i^2}).
\end{align*}
Since $A_i$ is symmetric for all $i\in \{1, 2, 3\}$, it has real eigenvalues and therefore
\begin{equation}\label{t-ineq}
    3t_i^2 \leq 4n - 1.
\end{equation}

Now, from Equation (6.9) in~\cite{pmn} we get
\begin{equation}\label{eqn-t}
    \begin{cases}
    t_1 + t_2 + t_3 = -1\\
    \sqrt{4n - 1 - 3t_1^2} + \sqrt{4n - 1 - 3t_2^2} + \sqrt{4n - 1 - 3t_3^2} = 0
    \end{cases}
\end{equation}
Now eliminating $t_3$ and rationalising gives us
\begin{align*}
    t_1^2(3t_2 + 2n + 1) + t_1(3t_2^2 + 2nt_2 + 4t_2 + 2n + 1)\\
    + (2n + 1)(t_2^2 + t_2) - 2n(n - 1) = 0.
\end{align*}
Notice that 
\[
3t_2^2 + 2nt_2 +4t_2 + 2n + 1 = (3t_2 + 2n + 1)(t_2 + 1).
\]
Therefore, $3t_2 + 2n + 1$ divides $(2n + 1)(t_2^2 + t_2) - 2n(n - 1)$.
Now consider the equation
\[
2n(2n + 1)(t_2^2 + t_2) - 4n(n - 1) \equiv 0 \pmod{3t_2 + 2n + 1}
\]
If we eliminate $n$ from the equation, we deduce that $3t_2 + 2n + 1$ must divide $3(t_2 + 1)^2(2t_2 + 1)$.

Notice that there is complete symmetry between $t_1, t_2$, and $t_3$. Hence, we deduce that
\begin{equation}\label{eqn-b}
  3(t_i + 1)^2(2t_i + 1) \equiv 0  \pmod{3t_i + 2n + 1}
\end{equation}
for all $i\in \{1, 2, 3\}$.

Using the equation for $\Tr(A_i^3)$ we deduce that $3nk_i$ must divide $k_i^3 + n(r_i^3 + s_i^3) + (n - 1)t_i^3$. Substitution for $k_i, r_i, s_i$ in terms of $t_i$ and algebraic manipulation gives
\begin{equation}
2n^2 - 6n - 6t_i^2 + 2t_i^3 + (1 + t_i)(4t_i^2 - t_i + 1) \equiv 0 \pmod{6(n + t_i)}. 
\end{equation}
Reducing modulo $2n + t_i$, we deduce that $2n^2 - 6n \equiv 2t_i(t_i + 3)$ and simplifying gives 
\begin{equation}\label{eqn-2b}
    (t_i + 1)(2t_i + 1)(3t_i + 1) \equiv 0 \pmod{2n + t_i}.
\end{equation}

Since $t_1 + t_2 + t+3 = -1$ and $t_i\in \Z$ for all $i\in \{1, 2, 3\}$, not all them can be negative. Let $b$ be one of them such that $b\geq 0$. Then, it follows by~\ref{eqn-b} and~\ref{eqn-2b} that
\begin{align*}
    (b + 1)(2b + 1)(3b + 1) = u.(2n + b)\\
    (b + 1)(2b + 1)(3b + 3) = v.(2n + 3b + 1)
\end{align*}
for some $u, v\in \Z$.
Now subtracting gives
\[2(b + 1)(2b + 1) = 2(v - u)(n + b) + v(b + 1)\].

Now set $w = v - u$. We want to show that $w = 0$. Firstly notice that 
\begin{align}\label{eqn-w}
    w = (b + 1)(2b + 1)\left(\frac{3b + 3}{2n + 3b + 1} - \frac{3b + 1}{2(n + b)}\right)\\
    = \frac{(b + 1)(2b + 1)(4n - 1 - 3b^2)}{2(2n + 3b + 1)(n + b)}
\end{align}
and hence, by Equation~\ref{t-ineq}, $w\geq 0$.
Rearranging gives 
\[
3(b + 1)^3(2b + 1) = (2n + 3b + 1)\left(2(b + 1)(2b + 1) - 2w(n + b)\right).
\]
Setting $n + b = x$ and refactorising we get the following quadratic in terms of $x$:
\[
4wx^2 - 2(n + 1)(4b + 2 - w)x + (b + 1)^2(2b + 1)(3b + 1) = 0.
\]

By definition $x$ is real and hence, the discriminant of this quadratic must be non-negative. Therefore,
\[
4(b + 1)^2(4b + 2 - w)^2 - 16w(b + 1)^2(2b + 1)(3b + 1)\geq 0
\]
and hence
\begin{align}\label{discr}
    (4b + 2 - w)^2 \geq 4w(2b + 1)(3b + 1)\\
    =w(4b + 2)(6b + 2).
\end{align}
By~\ref{eqn-w} we have that $w < 2b + 1$. Now since $w\geq 0$ it follows that $2b + 1 < 4b + 2 - w \leq 4b + 2$. Now, by~\ref{discr}, we get that $w\leq 0$ and hence $w = 0$, as claimed. Therefore, by~\ref{eqn-w} $4n - 1 = 3b^2$ and hence $b$ must be odd. We therefore set $b = 2a + 1$ for $a \geq 0$ and it follows that $n = 3a^2 + 3a + 1$. Now suppose without loss of generality that $t_1$ was $b$. Then from~\ref{eqn-t}
\[
    t_2^2 = t_3 ^ 2
\]
and therefore
\[
    t_2 = \pm t_3.
\]
But we know that $t_2 + t_3 = -1 -t_1 \neq 0$ and hence 
\[
t_2 = t_3 = \frac{-1 - t_1}{2}.
\]
Hence, $t_1 = 2a + 1, t_2 = t_3 = -a -1$.

Moreover, since we've shown that $v_i$ is odd, $a$ must be even and 
\[
k_1 = n + 2a + 1
\]
\[
k_2 = k_3 = n - a - 1
\]
as required.
\end{proof}

\subsection{Proof of Theorem~\ref{thm-typeV}}

\begin{proof}
Let $k_i, r_i, s_i, t_i$ be the eigenvalues of $A_i$ for $i\in \{1, \ldots, 5\}$ with multiplicities $1, n, n, n - 1$ respectively. If the matrices $\Theta_{i, 1}$ are as in~\cite{pmn}, then they must be $2\times 2$ matrices with eigenvalues $r_i, s_i$, where $r_i$ and $s_i$ are the eigenvalues of $A_i$ with multiplicity $n$. We know that $r_i, s_i$ must necessarily be rational integers. Now from the linear trace equation 
\[
\Tr(A_i) = k_i + n(r_i + s_i) + (n - 1)t_i
\]
we deduce that $n$ must divide $k_i - t_i$ and since by Lemma~\ref{lemma-perron-frob} $|t_i| < k_i$, it follows that $k_i = \epsilon_in + t_i$ for $\epsilon_i \geq 1$ for all $i$. Therefore, $\sum_{i = 1}^{5} \epsilon_i \geq 5$. On the other hand,
\[
3n - 1 = \sum_{i = 1}^5 k_i = (\sum_{i = 1}^5 \epsilon_i)n + \sum_{i = 1}^5 t_i = (\sum_{i = 1}^5 \epsilon_i)n - 1 
\]
and hence $\sum_{i = 1}^5 \epsilon_i = 3$, a contradiction. Therefore, this type of coherent configuration cannot exist.
\end{proof}

\subsection{Proof of Theorems~\ref{thm-typeVII} and~\ref{thm-typeVIII}}

In this section we deal with the cases arising in Theorems~\ref{thm-typeVII} and~\ref{thm-typeVIII} together. We prove both statements through a series of lemmas that eliminate the case arising in Theorem~\ref{thm-typeVIII} and force the parameters stated in Theorem~\ref{thm-typeVII}. 


\begin{lemma}\label{lemma-symmetric}
If $\mathcal{A} = \{A_1, A_2, A_3, A_4\}$ is a homogeneous coherent configuration of rank 4, where its matrices have eigenvalue multiplicities $1, n + 1, n - 1$, and $n - 1$, then all matrices are symmetric.
\end{lemma}

\begin{proof}
Suppose for a contradiction that this is not the case. Then, since $\mathcal{A}$ is a homogeneous coherent configuration, one of the matrices say $A_1$ must be symmetric and $A_2, A_3$ are such that $A_2^T = A_3$. Then, $A_2$ and $A_3$ would have the same eigenvalues. Let $k_i, r_i, s_i, t_i$ for $i\in \{1, 2, 3\}$ be the eigenvalues of $A_1, A_2, A_3$ respectively with multiplicities $1, n + 1, n - 1, n - 1$ respectively. Then, since $A_2 = A_3^T$, $A_2$ and $A_3$ have the same eigenvalues with the same multiplicities. Hence, $s_2 + t_2 = s_3 + t_3$. But then, since by Equation (6.9) in~\cite{pmn}
\begin{align*}
    s_1 + s_2 + s_3 = -1\\
    t_1 + t_2 + t_3 = -1
\end{align*}
it follows that $s_1 = t_1$. However, Theorem~\ref{thm-typeVI} such a matrix cannot exist, a contradiction. Therefore, all matrices of $\mathcal{A}$ must be symmetric.
\end{proof}

For the remainder of the section, given a coherent configuration $\mathcal{B}$ we consider the association scheme $\mathcal{A}$ arising by adding every non-symmetric matrix and its transpose together to make a symmetric matrix. In this case notice that if $B_i$ has eigenvalues $n_i, \lambda_i, \mu_i, \nu_i$ then $A_i = B_i + B_i^T$ has eigenvalues $k_i = 2n_i, r_i = 2\lambda_i, s_i = 2\mu_i, t_i = 2\nu_i$ again with eigenvalue multiplicities $1, n + 1, n-1, n - 1$ respectively.

\begin{lemma}\label{lemma-epsilon}
 If $\mathcal{A}$ is as defined above, then $k_i = \epsilon_i(n - 1) - 2r_i$ for some $\epsilon_i \leq 0$ for all $i$. Moreover, $\sum \epsilon_i = 3$.
\end{lemma}

\begin{proof}
By the linear trace relation for $A_i$ we get
\[
\Tr(A_i) = k_i + (n + 1)r_i + (n - 1)(s_i + t_i)
\]
Hence, $k_i \equiv -2r_i \pmod{n - 1}$ and we can write
\[
k_i = \epsilon_i(n - 1) -2r_i
\]
as claimed.

Also, notice that 
\[
3n - 1 = \sum_{i}k_i = (n - 1) \sum_{i} \epsilon_i - 2\sum_i r_i.
\]
Since by Equation (6.9) in~\cite{pmn} $\sum_{i} r_i = -1 $, it follows that $\sum_{i} \epsilon_i = 3$.

Now suppose for a contradiction that $\epsilon_i < 0$. 
Since $k_i \geq 0$ it follows that $r_i < 0$. In particular, since by Lemma~\ref{lemma-perron-frob} $|r_i| < k_i$ we have that $-r_i < (n - 1)\epsilon_i - 2r_i$ and hence $|r_i| > n-1$ and thus $|r_i| \geq n$. By the quadratic trace relation we get
\[
k_i^2 + (n + 1)r_i^2 \leq \Tr(A_i^2) = 3nk_i.
\]
Hence, $(n + 1)r_i^2 \leq k_i(3n - k_i)$, and basic calculus shows that $k_i(3n - k_i)$ is maximised at $k_i = \frac{2n}{2}$. Hence,
\[
nr_i < (n + 1)r_i \leq \left( \frac{3n}{2}\right)^2.
\] 

Dividing through by $n$ and applying sqare roots gives us $|r_i| < \frac{3\sqrt{n}}{2} < n$, a contradiction. Hence $\epsilon_i > 0 $ for all $i$.
\end{proof}

Now considering the quadratic trace equation again and reducing modulo $n - 1$ we get
\[
\Tr(A_i^2) = k_i^2 + (n + 1)r_i^2 + (n - 1)(\lambda_i^2 + \mu_i^2) = 3nk_i \Rightarrow
\]
\[
(-2r_i)^2 +2r_i^2 = -6r_i \Rightarrow 6r_i(r_i + 1) \equiv 0 \pmod{n - 1}.
\]

We now show that in fact $n - 1$ divides $3r_i(r_i + 1)$.

\begin{lemma}\label{lemma-divcond}
If $r_i$ is as defined above, then $n - 1$ divides $3r_i(r_i + 1)$.
\end{lemma}

\begin{proof}
The trace equations give 
\[
    s_i + t_i = -\epsilon_i - r_i
\]
\[
    (n - 1)(s_i^2 + t_i^2) = 3n(\epsilon_i(n - 1) - 2r_i) - (\epsilon_i(n - 1) -2r_i)^2 - (n + 1)r_i^2
\]
Now $s_it_i$ is a rational integer by assumption and also $2s_it_i = (s_i + t_i)^2 - (s_i^2 + t_i^2)$. Calculating modulo $2(n - 1)$ we get
\begin{align*}
    0 \equiv (n - 1)(\epsilon_i + r_i)^2 - 3n(\epsilon_i(n - 1) - 2r_i) - (\epsilon_i(n - 1) - 2r_i)^2 - (n + 1)r_i^2\\
    \equiv (n - 1)(\epsilon_i^2 + r_i ^2) - \epsilon_i(n - 1) + 6r_i + 4r_i^2 + (n - 1)r_i^2 + 2r_i^2\\
    \equiv (n - 1)(\epsilon_i^2 - \epsilon_i) + 6r_i + 6r_i^2.
\end{align*}
Since $\epsilon_i^2 - \epsilon_i$ is a product of consecutive integers it is even and hence $2(n - 1)$ must divide $6r_i(r_i + 1)$ and hence $n - 1$ divides $3r_i(r_i - 1)$, as claimed.
\end{proof}

We now prove another inequality that we will use later.

\begin{lemma}\label{lemma-largeineq}
$\epsilon_i(n - 1)(6n - 2\epsilon_in + \epsilon_i) - 6r_i(2n - \epsilon_in + \epsilon_i) - (3n + 9)r_i^2 \geq 0$
\end{lemma}

\begin{proof}
Consider the quadratic equation whose roots are $s_i$ and $t_i$. Since $s_i$ and $t_i$ are real, it follows that the discriminant of this equation, namely $(s_i + t_i)^2 - 4s_it_i = (s_i - t_i)^2$ is non-negative. Notice that $(s_i - t_i)^2 = 2(s_i^2 + t_i^2) - (s_i + t_i)^2$ and hence using the trace equations we get
\[
6n(\epsilon_i(n - 1) - 2r_i) - 2(\epsilon_i(n - 1) - 2r_i)^2 - 2(n - 1)r_i^2 - (n - 1)(\epsilon_i + r_i)^2 \geq 0.
\]
This can be rearranged to give the required statement.
\end{proof}

From Lemma~\ref{lemma-epsilon} we know that either one of the $\epsilon_i$s is zero say $\epsilon_1$ without loss of generality, or there are just three non-identity matrices and $\epsilon_1 = \epsilon_2 = \epsilon_3 = 1$. We first consider the former case.

\begin{prop}\label{prop-719}
If $\epsilon_1 = 0$, then $n = 7$ or $19$ and the coherent configurations are symmetric.
\end{prop}

\begin{proof}
If $\epsilon_1 = 0$, then $k_1 = -2r_1$ and since $k_1 >0$, it follows that $r_1 < 0$. Using Lemma~\ref{lemma-largeineq} we get
\[
-12nr_1 - (3n + 9)r_1^2 \geq 0
\]
and hence $r_1 \geq \frac{-4n}{n + 3} > -3$.

Therefore, $r_1 = -3$ or $r_1 = -2$, or $r_1 = -1$ and $k_1 = 6, 4$, or $2$. Consider the case where $k_1 = 2$ and $r_1 = -1$. The trace equations give us
\[
s_1 + t_1 = 1
\]
\[
(n - 1)(s_1^2 + t_1^2) = 5n - 5 \Rightarrow s_1^2 + t_1^2 = 5.
\]
Therefore, $s_1$ and $t_1$ are equal to $2$ and $-1$ respectively. However, $r_1 = -1$ and $k_1 = 2$ but by Lemma~\ref{lemma-perron-frob} $|s_1| < k_1$, a contradiction. Hence, $k_1 = 2$ cannot hold.

It now follows by Lemma~\ref{lemma-divcond} that $n - 1$ divides $18$ or $n - 1$ divides $6$. Using the inequality from Lemma~\ref{lemma-largeineq} we deduce that either $r_1 = -3$ and $n = 10$ or $n = 19$, or $r_1 = -2$ and $n = 3, 4$, or $7$.

Now define $A = \sum\{A_i \mid \epsilon_i = 0\}$. Then, $A$ must be a symmetric matrix of row sum $k = \sum k_i$ and eigenvalue $r = \sum r_i$. What we have said above for matrices $A_i$ with $\epsilon_i = 0$ applies to $A$ as well and therefore $A$ must consist of only one summand, $A_1$ without loss of generality. Now since by Lemma~\ref{lemma-epsilon} $\sum \epsilon_i = 3$ there are two possibilities. There are either 5 matrices and $\epsilon_2 = \epsilon_3 = \epsilon_4 = 1$ or there are 4 matrices and $\epsilon_2 = 2$ and $\epsilon_3 = 1$. 

Now we check this case individually to see which of those can hold. 

\medskip

\noindent \textbf{Case 1:} $r_1 = -2, n = 3$.

\medskip In the case that $\epsilon_2 = \epsilon_3 = \epsilon_4 = 1$ the inequality from Lemma~\ref{lemma-largeineq} gives us 
\[
13 - 12r_i - 9r_i^2 \geq 0
\]
for $i\in \{2, 3, 4\}$ and since $r_i$ is integer, $-3 \leq r_i \leq 0$. Since by Equation (6.9) in~\cite{pmn} $r_1, r_2, r_3, r_4$ must sum up to $-1$, it follows that $r_2, r_3, r_4$ must sum up to $1$, but this cannot hold since none of them can be positive. 

Now we examine the case where we have four matrices and $\epsilon_2 = 1$ and $\epsilon_3 = 2$. In this case Lemma~\ref{lemma-largeineq} gives us
\begin{align*}
    -3 \leq r_2 \leq 0\\
    -2 \leq r_3 \leq 1.
\end{align*}
The only way $r_2$ and $r_3$ could sum up to $1$ is $r_2 = 0$ and $r_3 = 1$. In this case we get $k_1 = 4, k_2 = 2, k_3 = 2$ and checking for such coherent configurations in~\cite{HanakiMiyamoto} we find that there is a unique coherent configuration with such row and column sums, but checking the rational eigenvalues using GAP~\cite{gap} shows that the $r_i$s are not equal to $-2, 0, 1$ as we wish and hence there is no such association scheme.

\medskip

\noindent \textbf{Case 2:} $r_1 = -2, n = 4$. 

\medskip

First we look at the case where $\epsilon_2 = \epsilon_3 = \epsilon_4 = 1$. By Lemma~\ref{lemma-largeineq} we get
\[
-7r_i^2 - 15r_i + 17 \geq 0
\]
Since $r_i$ is integer for $i\in \{2, 3, 4\}$ this gives
\[
-2 \leq r_i \leq 0
\]
Again in this case we want the $r_i$s for $i\in \{2, 3, 4\}$ to sum up to $1$ but none of them is positive, so this case cannot hold.

Now let $\epsilon_2 = 1$ and $\epsilon_3 = 2$. In this case Lemma~\ref{lemma-largeineq} gives 
\begin{align*}
    -2 \leq r_2 \leq 0\\
    -2 \leq r_3 \leq 1.
\end{align*}
The only combination that could work is $k_2 = 0$ and $k_3 = 1$. In this case we would get $k_1 = 4, k_2 = 3, k_3 = 4$. Checking in~\cite{HanakiMiyamoto} we don't find any coherent configurations with such row and column sums and appropriate eigenvalues and hence $n = 4$ cannot hold either.

\medskip

\noindent \textbf{Case 3:} $r_1 = -2, n = 7$.

\medskip 

In the case that $\epsilon_2 = \epsilon_3 = \epsilon_4 = 1$ Lemma~\ref{lemma-largeineq} gives us
\[
-15r_i^2 -16r_i + 58 \geq 0
\]
and hence, since $r_i\in \Z$ for $i\in \{2, 3, 4\}$,
\[
-2 \leq r_i \leq 1
\]
The only combinations (up to permutation) that would give us $r_1 + r_2 + r_3 + r_4 = -1$ are $r_2 = 1, r_3 = 0, r_4 = 0$ and $r_2 = -1, r_3 = 1, r_4 = 1$. We then get $k_1 = 4, k_2 = 4, k_3 = 6, k_4 = 6$ or $k_1 = 4, k_2 = 4, k_3 = 4, k_4 = 8$ respectively. Looking at~\cite{HanakiMiyamoto}, we deduce that there aren't any coherent configurations with such matrix row and column sums.

For $\epsilon_2 = 1, \epsilon_3 = 2$, as shown in~\cite{pmn} we need $k_1 = 4, k_2 = 8, k_3 = 8$ and looking at~\cite{HanakiMiyamoto} we deduce that there is a unique coherent configuration with such matrix row and column sums and hence, it is the one arising in~\cite{pmn}. The corresponding $s_i$s and $t_i$s can be calculated to be
\begin{align*}
    s_1 = 1 + \sqrt{2}, t_1 = 1 - \sqrt{2}\\
    s_2 = -2\sqrt{2}, t_2 = 2\sqrt{2}\\
    s_3 = -2 + \sqrt{2}, t_3 = -2 -\sqrt{2}.
\end{align*}

\medskip

\noindent \textbf{Case 4:} $r_1 = -3, n = 10$.

\medskip

In this case it suffices to check the subcase $\epsilon_2 = 1, \epsilon_3$, since $r_1$ is odd and hence it cannot be the case that $A_1$ is the sum of a matrix and its transpose. Therefore, all the matrices in the initial coherent configuration must be symmetric and we must have four of them. In this case by Lemma~\ref{lemma-largeineq} we get
\begin{align*}
    -13r_2^2 - 33r_2 + 123 \geq 0\\
    -39r_3^2 - 2r_3 + 396 \geq 0
\end{align*}
which gives
\begin{align*}
    -4 \leq r_2 \leq 2\\
    -3 \leq r_3 \leq 3.
\end{align*}
The $(r_2, r_3)$ pairs consistent with Equation (6.9) in~\cite{pmn} are $(2, 0), (1, 1), (0, 2), (-1, 3)$ and all of those give row and column sums for which an association scheme does not exist.

\medskip

\noindent \textbf{Case 5:} $r_1 = -3, n = 19$.

\medskip In this case, as shown in~\cite{pmn} $k_1 = 6, k_2 = 20, k_3 = 30$ and the corresponding $s_i$s and $t_i$s are
\begin{align*}
    s_1 = \frac{3 + \sqrt{5}}{2}, t_1 = \frac{3 - \sqrt{5}}{2}\\
    s_2 = -2\sqrt{5}, t_2 = 2\sqrt{5}\\
    s_3 = \frac{-5 + 3\sqrt{5}}{2}, t_3 = \frac{-5 -3\sqrt{5}}{2}.
\end{align*}
\end{proof}

We now deal with the case where $\epsilon_1 = \epsilon_2 = \epsilon_3 = 1$. Notice that in this case, since the $\epsilon_i$s are all odd, $\mathcal{B} = \mathcal{A}$ and by Lemma~\ref{lemma-symmetric} all matrices are symmetric.

\begin{lemma}\label{lemma-rdiff}
If $\epsilon_1 = \epsilon_2 = \epsilon_3 = 1$ then $r_1, r_2, r_3$ are all different.
\end{lemma}

\begin{proof}
Suppose for a contradiction that this is not the case and without loss of generality, let $r_1 = r_2$. Then, since $\epsilon_1 = \epsilon_2$, it follows that $k_1 = k_2$. Thus, either $s_1 = s_1$ and $t_1 = t_2$ or $s_1 = t_2$ and $s_2 = t_1$, and since our coherent configuration has rank $4$, the matrices are simultaneously diagonalisable and it follows that 
\begin{align*}
    s_1 + s_2 + s_3 = -1\\
    t_1 + t_2 + t_3 = -1.
\end{align*}
But this means that $s_3 = t_3$ and thus $A_3$ is a matrix of the kind that Theorem~\ref{thm-typeVI} forbids, a contradiction.
\end{proof}

\begin{lemma}\label{lemma-aineq}
Let $a_i = \frac{3r_i(r_i + 1)}{n - 1}$. Then, $a_i \leq 4$ and if $r_i \geq 0$, then $a_i \leq 3$.
\end{lemma}

\begin{proof}
Firstly notice that by Lemma~\ref{lemma-divcond}, $a_i\in \Z$. By Lemma~\ref{lemma-largeineq} we get
\[
(n - 1)(4n + 1) - 6(n + 1)r_i - (3n + 9)r_i^2 \geq 0.
\]
Therefore, 
\[
(3n + 9)(r_i^2 + r_i) \leq (n - 1)(4n + 1) - (3n - 3)r_i
\]
and hence
\begin{align*}
    a_i = \frac{3r_i(r_i + 1)}{n - 1} \leq \frac{4n + 1}{n + 3} - \frac{3r_i}{n + 3}\\
    = 4 - \frac{11}{n + 3} - \frac{3r_i}{n + 3}.
\end{align*}

Now, if $r_i \geq 0$, we get $a_i < 4$, and hence $a_i \leq 3$. If $r_i < 0 $ and $n \geq 19$, using the inequality from Lemma~\ref{lemma-epsilon} stating that $r_i < \frac{3\sqrt{n}}{2}$, we deduce that $\frac{-3r_i}{n + 3} \leq 1$ and hence $a_i < 5$ and so $a_i \leq 4$. Now, if $n < 19$ and $r_i \leq 0$ checking gives that $a_i \leq 3$.
\end{proof}

\begin{lemma}\label{lemma-adiff}
If none of $a_1, a_2, a_3$ are zero, then $a_1, a_2, a_3$ are all different.
\end{lemma}

\begin{proof}
Suppose for a contradiction that without loss of generality, $a_1 = a_2$. Then, both $r_1$ and $r_2$ are roots of the equation
\[
3r(r + 1) - a_1(n - 1) = 0.
\]
Since by Lemma~\ref{lemma-rdiff} $r_1 \neq r_2$, we must have $r_1 + r_2 = -1$. But from Equation (6.9) in~\cite{pmn}, $r_1 + r_2 + r_3 = -1$ and hence $r_3 = 0$. But then, $a_3 = 0$, a contradiction.
\end{proof}

\begin{lemma}\label{lemma-root}
If $a > 0$ and $r$ is a root of the equation
\[
x^2 + x - a = 0
\]
then $r = -\frac{1}{2} \pm \sqrt{a} + \eta$, where $|\eta| < \frac{1}{8\sqrt{a}}$.
\end{lemma}
\begin{proof}
Notice that $\left( r + \frac{1}{2}\right)^2 = r^2 + r + \frac{1}{4} = a + \frac{1}{4}$.

Now, by squaring both $\sqrt{a + \frac{1}{4}}$ and $\sqrt{a} + \frac{1}{8\sqrt{a}}$ we see that $|\eta| < \frac{1}{8\sqrt{a}}$, as claimed.
\end{proof}

\begin{lemma}
One of $a_1, a_2, a_3$ must be zero.
\end{lemma}

\begin{proof}
Suppose that this is not the case. Then, by Lemma~\ref{lemma-adiff}, $a_1, a_2, a_3$ are all different. Since $a_i = \frac{3r_i(r_i + 1)}{n - 1}$, it follows that $r_i$ is a root of the equation 
\[
x^2 + x -\frac{a_i(n - 1)}{3} = 0.
\]
By Lemma~\ref{lemma-root} we get that 
\[
r_i = -\frac{1}{2} \pm \sqrt{\frac{a_i(n - 1)}{3}} + \eta_i
\]
where $|\eta_i| < \frac{1}{8}\sqrt{\frac{3}{a_i(n - 1)}} < \frac{1}{8}$.

Now, it follows by Equation (6.9) in~\cite{pmn} that 
\[
r_1 + r_2 + r_3 = -1
\]
and hence
\[
-\frac{3}{2} + \sqrt{\frac{n - 1}{3}}(\pm \sqrt{a_1} \pm \sqrt{a_2} \pm \sqrt{a_3}) + \eta_1 + \eta_2 + \eta_3 = -1.
\]
Rearranging and taking absolute values gives
\[
\left| \sqrt{\frac{n - 1}{3}}(\pm \sqrt{a_1} \pm\sqrt{a_2} \pm \sqrt{a_3}) \right|< \frac{7}{8}.
\]
Since $a_i \neq 0$, by Lemmas~\ref{lemma-aineq} and~\ref{lemma-adiff} we get that $a_1, a_2, a_3$ must be among the numbers $1, 2, 3, 4$ and all different. Hence, crude approximations to $sqrt{2}$ and $\sqrt{3}$ give the estimate
\[
|\pm \sqrt{a_1} \pm \sqrt{a_2} \pm \sqrt{a_3}| > \frac{4}{10}
\]
and hence 
\[
\frac{4}{10}\sqrt{\frac{n - 1}{3}} < \frac{7}{8}
\]
This gives $n < 15$, but checking all cases shows that no integer less than $15$ has three different representations in the form $1 + \frac{3r_i(r_i + 1)}{a_i}$ with $r_i, a_i$ integral, all different for every $i$, and $1 \leq a_i \leq 4$, a contradiction. Hence, one of $a_1, a_2, a_3$ must be zero, as claimed.
\end{proof}

We now choose notation such that $a_1 = 0$.

\begin{lemma}
If $r_1$ is as defined above, then $r_1 = -1$.
\end{lemma}

\begin{proof}
Since $a_1 = 0$, $r_1 = 0$ or $r_1 = -1$. Assume now that $r_1 = 0$. One of $a_2, a_3$ must be zero, for otherwise, all $r_i$s would be solutions of the equation $x^2+ x = 0$ and hence they would not all be different, as Lemma~\ref{lemma-rdiff} states. Suppose without loss of generality that $a_2 \neq 0$. Then,
\[
n = \frac{3r_2(r_2 + 1)}{a_2} + 1
\]
If $3$ does not divide $a_2$ then $n\equiv 1 \pmod{3}$. If $3$ divides $a_2$ then by Lemma~\ref{lemma-aineq}, it follows that $a_2 = 3$ and $n = r_2^2 + r_2 + 1$. Hence $n\equiv 1 \pmod{3}$ or $n\equiv 0 \pmod{3}$. From the linear and quadratic trace equations for $A_1$ we get 
\begin{align*}
    s_1 + t_1 = -1\\
    s_1^2 + t_1^2 = 3n.
\end{align*}
Now $p_{11}^1 = |\{j\in \{1, \ldots, 3n\} \mid (A_1)_{ij} = 1, (A_1)_{jk} = 1\}|$ is an integer constant for any $i, k\in \{1, \ldots, 3n\}$ such that $(A_1)_{ik} = 1$. Moreover, the cubic trace equation for $A_1$ gives
\begin{align*}
    3np_{11}^1 = (n - 1)^2 + \frac{3}{2}(s_1 + t_1)(s_1^2 + t_1^2) - \frac{1}{2}(s_1 + t_1)^3\\
    = (n - 1)^2 - \frac{3}{2}(2n + 1) + \frac{1}{2}\\
    = n^2 - 5n.
\end{align*}
Thus, $3p_{11}^1 = n - 5$ and since $p_{11}^1\in \Z$, it follows that $n \equiv 5 \pmod{3}$, a contradiction. Hence $r_1 \neq 0$ and therefore $r_1 = -1$.
\end{proof}

\begin{lemma}\label{lemma-31}
$n = 31$.
\end{lemma}

\begin{proof}
Since $r_1 = -1$ and $r_1 + r_2 + r_3 = -1$ by Equation (6.9) in~\cite{pmn}, it follows that $r_2 = -r_3 = r\in \Z$. Then, since $a_2, a_3$ are integers,
\begin{align*}
    n - 1 \text{ divides } 3r(r + 1)\\
    n - 1 \text{ divides } 3(-r)(-r + 1).
\end{align*}
Hence, $n - 1$ divides $3r(r + 1) - 3(-r)(-r + 1) = 6r$.

By interchanging $A_2$ and $A_3$ if necessary we may assume that $r \geq 0$. Then since by Lemma~\ref{lemma-rdiff}, $r_1, r_2, r_3$ are all different, it follows that $r\neq 0$ and $r\neq 1$. Hence, $r \geq 2$. Moreover, from Lemma~\ref{lemma-aineq} we know that
\[
\frac{6r}{n - 1}\cdot \frac{r + 1}{2} \leq 3.
\]
It follows that $r + 1 \leq 6$ and if $6r\neq n - 1$ then since $n - 1$ divides $6r$, $r + 1\leq 3$. Now considering that $\frac{6r}{n - 1}\cdot \frac{r + 1}{2}$ must be integer and that the above inequality must hold for our choices of $n$ and $r$ we can check all cases and find that the only possibilities are:
\begin{align*}
    6r = n - 1, r = 5, n = 31\\
    6r = n - 1, r = 3, n = 19\\
    3r = n - 1, r = 2, n = 7.
\end{align*}
For $n = 7$ we see that $k_1, k_2, k_3$ are equal to $8, 2, 10$ respectively and checking in~\cite{HanakiMiyamoto}, we see that there is no association scheme with such row and column sums.

If $n = 19$, then the trace equations give 
\begin{align*}
    s_1, t_1 \text{ are } \pm 2\sqrt{5}\\
    s_2, t_2 \text{ are } \pm -2 \pm \sqrt{6}\\
    s_3, t_3 \text{ are } \pm 5, -3
\end{align*}
Now, no possible tuple $(s_1, s_2, s_3)$ satisfies $s_1 + s_2 + s_3 = -1$ and hence this case cannot arise.

Finally, for $n = 31$ for suitable choices of roots we get
\begin{align*}
    s_1 = 4\sqrt{2}, t_1 = -4\sqrt{2}\\
    s_2 = -3 -\sqrt{2}, t_2 = -3 + \sqrt{2}\\
    s_3 = 2 - 3\sqrt{2}, t_3 = 2 + 3\sqrt{2}.
\end{align*}
\end{proof}


\begin{proof}[Proof of Theorem~\ref{thm-typeVII}]
Follows directly by Proposition~\ref{prop-719} and Lemma~\ref{lemma-31}.
\end{proof}

\begin{proof}[Proof of Theorem~\ref{thm-typeVIII}]
Follows directly by proposition~\ref{prop-719}.
\end{proof}

\section{Examples}
In this section we provide examples with the parameters found in Theorems ~\ref{thm-typeI} to ~\ref{thm-typeVIII}, in cases where they are known to exist.

\subsection{Theorem~\ref{thm-typeI}}

The classic examples of symmetric conference graphs are the Paley graphs. The vertex set of such a graph is the set of elements of a finite field whose order is congruent to $1$~(mod~$4$), and two vertices are connected by an edge if and only if their difference is a square in the field.

Similarly, the classic examples of doubly regular tournaments are the Paley
tournaments; the vertex set is the set of elements of a finite field of order
congruent to $3$~(mod~$4$), wich an arc from $a$ to $b$ if $b-a$ is a square.

\subsection{Theorem~\ref{thm-typeII}}

For the second set of parameters arising in Theorem~\ref{thm-typeII}, a known example (with $a=0$) is the triangular graph $T(6)$ and its complement; no further examples are known. For the other sets of parameters, no known example with fewer than 512 vertices is known.
Moreover, due to the large number of vertices that the given parameters force, it would be very hard to construct one.

\subsection{Theorem~\ref{thm-typeIII}}

For the first set of parameters arising in Theorem~\ref{thm-typeIII} and for $a\geq 2$, the graphs arising from Steiner systems of the type $S(2, a+1, n)$ with
$a\in\{1,2,3\}$ are known examples. The number of non-isomorphic Steiner
systems $(2,3,19)$ is $11,084,874,829$ (see \cite{ko}); these give pairwise
non-isomorphic graphs.
There is no known example of graphs with the second set of parameters, and the nonexistence in the case $a=2$ has been shown by Wilbrink and Brouwer~\cite{wb}.

\subsection{Theorem~\ref{thm-typeIV}}

We do not have any examples for this theorem. Is it possible to take a graph
of the type arising in Theorem~\ref{thm-typeIII}, and either split the edges
into two classes or put directions on the edges so as to form a coherent
configuration?

\subsection{Theorem~\ref{thm-typeVII}}

The cases $n=21$ and $n=57$ are realised by the groups $\mathrm{PGL}(2,7)$
and $\mathrm{PSL}(2,19)$ respectively. These can be found in the
\textsf{GAP}~\cite{gap} database of primitive permutation groups as
\texttt{PrimitiveGroup(21,1)} and \texttt{PrimitiveGroup(57,1)} respectively.

The database~\cite{HanakiMiyamoto} gives the basis matrices for the first
of these, and certifies its uniqueness. In the second case, the association
scheme is also known to be unique~\cite{cd}; the graph of valency~$6$ is the
distance-transitive \emph{Perkel graph} \cite{perkel}. Existence
in the final case with $93$ points is undecided, as far as we know.

\paragraph{Acknowledgement}
The research of the first author was supported by an Undergraduate Research
Bursary, number XCLM18, from the London Mathematical Society. The second
author acknowledges the Isaac Newton Institute for Mathematical Sciences,
Cambridge, for support and hospitality during the programme
\textit{Groups, representations and applications: new perspectives}
(supported by \mbox{EPSRC} grant no.\ EP/R014604/1), where he held a Simons
Fellowship.

\end{document}